\documentclass[reqno]{amsart}
\usepackage{amsfonts}
\usepackage{latexsym}
\usepackage{amssymb}
\usepackage{amsmath}

\usepackage{bbm}
\usepackage{color}
\usepackage{calrsfs}

\usepackage{enumerate}
\usepackage[inline]{asymptote}

\newcommand{\R}{\mathbb R}
\newcommand{\N}{\mathbb N}

\newcommand{\E}{\mathbb E}
\newcommand{\Pro}{\mathbb P}
\newcommand{\dif}{\,\mathrm{d}}
\newcommand{\ext}{\mathrm{ext}}
\newcommand{\con}{\mathrm{conv}}
\DeclareMathOperator*{\kmax}{k-max}

\DeclareMathOperator*{\ave}{Ave}

\newtheorem{thm}{Theorem}[section]
\newtheorem{cor}[thm]{Corollary}
\newtheorem{lemma}[thm]{Lemma}

\newtheorem{proposition}[thm]{Proposition}

\theoremstyle{remark}

\numberwithin{equation}{section}
\begin{document}

%%%%%%%%%%%%%%%%%%%%%%%%%%%%%%%%%%%%%%%%%%%%% 5
\title[Estimating averages of order statistics]{Estimating averages of order statistics of bivariate functions}

\author{Richard Lechner}
\address{Institute of Analysis, Johannes Kepler University Linz,
  Altenbergerstrasse 69, 4040 Linz, Austria} \email{richard.lechner@jku.at}

\author{Markus Passenbrunner}
\address{Institute of Analysis, Johannes Kepler University Linz,
  Altenbergerstrasse 69, 4040 Linz, Austria} \email{markus.passenbrunner@jku.at}

\author{Joscha Prochno}
\address{Institute of Analysis, Johannes Kepler University Linz,
  Altenbergerstrasse 69, 4040 Linz, Austria} \email{joscha.prochno@jku.at}

\date{\today}

\begin{abstract}
  We prove uniform estimates for the expected value of averages of order statistics of bivariate
  functions in terms of their largest values by a direct analysis.
  As an application, uniform estimates for the expected value of averages of order statistics of
  sequences of independent random variables in terms of Orlicz norms are obtained.
  In the case where the bivariate functions are matrices, we provide
  a ``minimal'' probability space which allows us to $C$-embed certain Orlicz spaces $\ell_M^n$
  into $\ell_1^{cn^3}$, $c,C>0$ being absolute constants.
\end{abstract}

\maketitle

\tableofcontents

\section{Introduction and main results} \label{sec:introduction}

In the series of papers \cite{GLSW1,GLSW2,GLSW3,GLSW4,GLSW5}, sequences of random variables and
their order statistics were studied in several different settings and the obtained results were
applied successfully to the local theory of convex bodies. In \cite{GLSW1},
the authors studied expressions of the form
\begin{equation} \label{eq:k decreasing rearrangement random variables}
  \E\sum_{k=1}^{\ell} \kmax_{1\leq i \leq n} |x_iX_i|, \qquad 1\leq \ell \leq n,
\end{equation}
with independent identically distributed (iid) random variables $X_i$, $i=1,\dots,n$ and real numbers $x_i$, $i=1,\dots,n$. Here, $\kmax_{1\leq i \leq n}X_i(\omega)$ is the $k$-th order statistic of a statistical sample of size $n$, which is equal to its $k$-th largest value. Besides being fundamental tools in statistics with various applications in, e.g., compressed sensing \cite{CRT,LZ}, wireless networks \cite{CCR}, and data streams \cite{ZLYKZY}, order statistics of random samples appear naturally in Banach space theory, e.g., in computations of the distribution of eigenvalues of random matrices \cite{MP, R}, and in calculating sharp bounds for the expectation of the supremum of Gaussian processes indexed by certain interpolated bodies \cite{GLMP}. For general information on order statistics, we refer the reader to \cite{DN}.

Especially we would like to point out that the results that were obtained in \cite{GLSW1}, in particular the estimates for \eqref{eq:k decreasing rearrangement random variables}, were crucial to obtain estimates for various parameters associated to the local theory of convex bodies \cite{GLSW2}, e.g., type and cotype constants, $p$-summing norms, volume ratios, and projection constants.

An integral tool in \cite{GLSW1,GLSW2} and thus in  \cite{GLSW3,GLSW4,GLSW5} are combinatorial estimates going back to S.~Kwapie\'n and C.~Sch\"utt \cite{KS1,KS2}. Those estimates relate an average (over the group of permutations) of the largest order statistic of a matrix $a$ to the average of its largest entries. To be more precise, it was shown that for all $a\in\R^{n\times n}$
\[
\frac{1}{n!} \sum_{\pi\in\mathfrak S_n} \max_{1\leq i \leq n} |a_{i\pi(i)}| \simeq \frac{1}{n}\sum_{k=1}^n s(k),
\]
where $s(k)$ is the $k$-th largest entry of the matrix consisting of the absolute values of $a$ and
$\mathfrak S_n$ is the symmetric group on $\{1,\dots,n\}$.
In \cite{LPP}, this was established in the following setting: under some assumptions on the normalized counting measure $\Pro$ on a collection
$G$ of maps from $I=\{1,\dots,n\}$ to $J=\{1,\dots,N\}$, we have that for every matrix
$a\in \mathbb R^{n\times N}$ and every $\ell \leq n$,
\begin{equation} \label{eq:main discrete case}
  \frac{c}{N} \sum_{j=1}^{\ell N}s(j)
  \leq \int_{G} \sum_{k=1}^{\ell}\kmax\limits_{1\leq i \leq n} |a_{ig(i)}| \dif \Pro(g)
  \leq \frac{C}{N} \sum_{j=1}^{\ell N}s(j),
\end{equation} 
where $c$ and $C$ are positive constants only depending on $G$. Special choices for $G$ so that \eqref{eq:main discrete case} holds, include the symmetric group $\mathfrak S_n$ and $\{1,\dots,n\}^{\{1,\dots,n\}}$. Those estimates were then used to deduce similar combinatorial estimates for $\ell_p$ norms.

In this work we extend our results from \cite{LPP} and study averages of order statistics of
bivariate functions $a:\{1,\dots,n\}\times \Omega \to \R$, where $(\Omega,\mathcal F,\mu)$ is an
arbitrary probability space.
In this setting, $G$ will be a collection of maps from $I=\{1,\dots,n\}$ to the probability space
$(\Omega,\mathcal F, \mu)$.
We denote the decreasing rearrangement of $a$ by $a^*$.
Our main result is the following:

\begin{thm}\label{thm:main}
  Let $n\in\N$, $G$ be a collection of maps from $I=\{1,\dots,n\}$ to the probability space
  $(\Omega,\mu)$, $C_G\geq 1$ be a constant only depending on $G$, and $\Pro$ be a probability
  measure on $G$.
  Assume that for all $i\in I$, all different indices $i_1,i_2\in I$ and all measurable sets
  $A,A_1,A_2\subset \Omega$,
  \begin{enumerate}[(i)]
  \item\label{eq:condition 1} $\Pro(g(i)\in A)=\mu(A)$,
  \item\label{eq:condition 2} $\Pro(g(i_1)\in A_1, g(i_2)\in A_2 ) \leq C_G \mu(A_1)\mu(A_2)$.
  \end{enumerate}
  Then, for every measurable function $a:I\times \Omega\to\mathbb{R}$ and for every $\ell \leq n$, 
  \begin{equation} \label{eq:main result firs part}
    c\cdot \int_0^\ell a^*(t)\dif t \leq \int_{G}  \sum_{k=1}^{\ell}\kmax\limits_{1\leq i \leq n} |a(i,g(i))| \dif \Pro(g) \leq C\cdot \int_0^\ell a^*(t)\dif t,
  \end{equation} 
  where $1/c=48(1 + 2C_G)^2$ and $C = 6 (1 + 2 C_G)$.
\end{thm}

As a direct consequence, when the bivariate functions are matrices and $\Pro$ is the normalized counting measure on $G$, we obtain one of the main results in \cite[Theorem 1.1]{LPP}:

\begin{cor}\label{thm:main LPP}
Let $n,N\in\N$ and $a\in \R^{n\times N}$. Let $G$ be a collection of maps from $I=\{1,\dots,n\}$ to $J=\{1,\dots,N\}$ and $C_G>0$ be a constant only depending on $G$. Assume that for all $i\in I$, $j\in J$ and all different pairs $(i_1,j_1),(i_2,j_2)\in I\times J$ 
\begin{enumerate}[(i)]
\item\label{eq:condition 1 cor} $\Pro(\{g\in G: g(i)=j\})=1/N$,
\item\label{eq:condition 2 cor} $\Pro(\{g\in G: g(i_1)=j_1, g(i_2)=j_2 \}) \leq C_G/N^2$.
\end{enumerate}
Then, for every $\ell \leq n$, 
\begin{equation} \label{eq:main result firs part cor}
  \frac{c}{N} \sum_{j=1}^{\ell N}s(j) \leq \int_{G} \sum_{k=1}^{\ell}\kmax\limits_{1\leq i \leq n} |a_{ig(i)}| \dif \Pro(g) \leq \frac{C}{N} \sum_{j=1}^{\ell N}s(j),
\end{equation} 
where $1/c = 48(1+2C_G)^{2}$ and $C=6(1+2C_G)$.
\end{cor}

In this work we also present an example of a set of maps, say $G_0$, with a minimal number of elements satisfying conditions \eqref{eq:condition 1 cor} and \eqref{eq:condition 2 cor} in Corollary \ref{thm:main LPP}, thus guaranteeing that the inequalities in \eqref{eq:main result firs part cor} hold. When $N=n$, the cardinality of $G_0$ is $n^2$. The set of maps provided here is based on finite fields of $n$ elements, where $n$ is a power of a prime number. It is not hard to see that if $G$ and $\mathbb P$ satisfy conditions~\eqref{eq:condition 1 cor} and~\eqref{eq:condition 2 cor} of Corollary~\ref{thm:main LPP} for some constant $C_{G} \geq 1$, then $G$ consists of at least $\frac{n^2}{C_{G}}$ elements.

We then apply this result to obtain the following:

\begin{thm}\label{thm:application}
There exist constants $c,C>0$ such that for all $n\in\N$ and every strictly convex, twice differentiable Orlicz function $M:[0,\infty)\to [0,\infty)$ that is strictly $2$-concave and satisfies $M^*(1)=1$, we have that $\ell_M^n\stackrel{C}{\hookrightarrow} \ell_1^{cn^3}$.
\end{thm}

We also provide an application of Theorem~\ref{thm:main} to sequences of iid random variables to obtain
estimates for \eqref{eq:k decreasing rearrangement random variables}. Those estimates are in terms
of Orlicz norms and recover Corollaries $2$ and $3$ of \cite{GLSW1}. To be more precise, we prove the following:

\begin{thm}\label{thm:main sequences of random variables}
  Let $X_1,\dots,X_n$ be a sequence of iid random variables with $\E |X_1|<\infty$. Let $1\leq \ell \leq n$ and $M$ be the N-function given by
  \begin{equation} \label{eq:definition M star}
    M^*\Big(\int_0^\beta X^*(z) \dif z \Big) = \frac{\beta}{\ell},\qquad 0\leq \beta \leq 1.
  \end{equation}
  Then, for all $x\in\R^n$,
  \[
  c \|x\|_M \leq \E \sum_{k=1}^{\ell} \kmax_{1\leq i \leq n} |x_iX_i| \leq C \|x\|_M,
  \]
  where $c,C>0$ are absolute constants.
\end{thm}

Let $M^*$ be given as in \eqref{eq:definition M star}. Then, for all $s\geq 0$,
\begin{equation*} %\label{eq:defM}
  M(s) = \int_0^s \int_{|X| \geq 1/(t\ell)} |X| \dif \Pro \dif t.
\end{equation*}
For $\ell=1$, this was shown in \cite[pp. 4-5]{GLSW5}. A simple calculation shows that it holds for
general $\ell$ as well.
Therefore, we indeed obtain Corollaries $2$ and $3$ of \cite{GLSW1}.

While in~\cite{GLSW1} the proof involves estimates for the largest order statistic of a matrix and
makes use of combinatorial results of \cite{KS1,KS2} in a crucial way, our approach is based on a
purely probabilistic and direct analysis of~\eqref{eq:k decreasing rearrangement random variables}
(Theorem~\ref{thm:main}), and is interesting in its own right.

The organization of the paper is as follows. Section \ref{sec:prel} serves the purpose of introducing notations and preliminary results that we use throughout the paper, where the measure theoretic ones are especially important for the proof of the main theorem. Section \ref{sec: proof main} contains the proof of Theorem~\ref{thm:main}. This is done by reducing the problem to the case of functions only taking values in $\{0,1\}$ and showing the result for this subclass of functions. 
Section~\ref{sec:applications} contains the application of Theorem~\ref{thm:main} to sequences of iid random variables and thus the proof of Theorem~\ref{thm:main sequences of random variables}.
In Section~\ref{sec:minimal set of maps}, we present the minimal set of maps $G_0$ needed to guarantee the inequalities in \eqref{eq:main result firs part cor} and give a proof of the embedding result Theorem \ref{thm:application} .

% % % % % % % % % % % % % % % % % % % % % 
% % % % % % % % % % % % % % % % % % % % % 
\section{Notation and preliminaries}\label{sec:prel}
% % % % % % % % % % % % % % % % % % % % % 
% % % % % % % % % % % % % % % % % % % % % 

Given a random variable $X:\Omega\to\R$ on a measure space $(\Omega,\mathcal A,\mu)$, we define its distribution function $F_X(t):= \mu(\{\omega\in\Omega: |X(\omega)| > t\})$, $t \geq 0$. The decreasing rearrangement of $X$ is then defined for all $t\geq 0$ by
\[
X^*(t) = \inf \{s\geq 0: F_X(s) \leq t \},
\]
where we use the convention that $\inf \emptyset = \infty$. Note that if $F_X$ is continuous and strictly decreasing, then $X^*$ is simply the inverse of $F_X$. Moreover, notice that $X$ and $X^*$ are equimeasurable, i.e., 
\[
\mu(X\in A)=\lambda(X^*\in A)
\]
for all measurable subsets $A\subset \mathbb R$, where $\lambda$ denotes the Lebesgue measure. By $\mathbbm 1_A$ we denote the characteristic function of a set $A$.

% We will write $A(t) \simeq B(t)$ to denote that there are absolute constants $c_1$ and $c_2$ such that $c_1 A(t) \leq B(t) \leq c_2 A(t)$ for all $t$, where $t$ denotes all implicit and explicit dependencies that the expressions $A$ and $B$ might have.
% If the absolute constants depend on a certain parameter $p$, we denote this by $\simeq_p$. By $c,C...$ we denote positive absolute constants and we write $c_p,C_p$ if they depend on some parameter $p$. The value of the constants may change from line to line.
A convex function $M:[0,\infty)\rightarrow[0,\infty)$ is called an Orlicz function, if $M(0)=0$ and if $M$ is not constant. An Orlicz function (as we define it) is bijective and continuous on $[0,\infty)$. 
Given an Orlicz function $M$, we define its conjugate function $M^*$ via the Legendre transform,
\[
M^*(x) = \sup_{t\in[0,\infty)}\big(xt-M(t)\big).
\]
We have $M=M^{**}$, since the Legendre transform is an involution. Note that $M^*$ is again an Orlicz function if $M$ is an $N$-function, i.e., if additionally
\[
\lim_{x\to 0}M(x)/x = 0\qquad\text{and}\qquad \lim_{x\to \infty}M(x)/x=\infty.
\]
Given an Orlicz function $M$ and a measure space $(\Gamma,\psi)$, the Orlicz space $L_M(\Gamma)$ is the space of all (equivalence classes of) measurable, real valued functions $f$ on $\Gamma$ such that
\[
\int_{\Gamma} M(|f|/\lambda) \dif\psi <\infty,
\]
for some $\lambda>0$. We equip $L_M(\Gamma)$ with the Luxemburg norm
\[
\|f\|_M = \inf\bigg\{\lambda >0 : \int_{\Gamma} M(|f|/\lambda) \dif\psi\leq 1\bigg\}.
\]
The closed unit ball of the space $L_M$ will be denoted by $B_M$.
Note also that if $M$ is an $N$-function, we have
\begin{equation}\label{eq:dualOrlicz}
  \|f\|_{M}\leq \sup_{g\in B_{M^*}} \int_{\Gamma} f\cdot g \dif \psi \leq 2\|f\|_M.
\end{equation}
For a detailed and thorough introduction to Orlicz spaces, cf. eg. \cite{RR} or \cite{LT77}.

Another result we will use is Paley-Zygmund's inequality:
\begin{thm}[Paley-Zygmund] \label{thm:paley-zygmund}
  For every non-negative random variable $Z$ and every number $0<\theta<1$, we have
  \begin{equation*} %\label{ine:paley zygmund}
    \Pro(Z\geq \theta \cdot\E Z)\geq (1-\theta)^2\frac{(\E Z)^2}{\E Z^2}.
  \end{equation*}	
\end{thm}

Moreover, we will need the following measure theoretic results:

\begin{thm}[Sierpi\'nski]\label{thm:sierpinski}
  Let $(R,\mathcal R,\rho)$ be a non-atomic measure space with $\rho(R)=c$. Then there exists a function $f:[0,c]\to \mathcal R$ satisfying 
  \begin{enumerate}[(i)]
  \item $f(t)\subset f(s)$ for $0\leq t\leq s\leq c$, 
  \item $\rho(f(s))=s$ for $0\leq s\leq c$.
  \end{enumerate}
\end{thm}

Sierpi\'nski's theorem allows us to construct to a given measurable function a new one that is constant only on sets of measure zero and which has the same ordering. 

\begin{proposition}\label{pro:construction-b}
  Let $(R,\mathcal R,\rho)$ be a finite measure space.
  For every measurable function $a:R\to[0,\infty)$ there exists a measurable function
  $b:R\to [0,\infty)$ with the following properties:
  \begin{enumerate}[(i)]
  \item \label{it:b1} for all $x\in[0,\infty)$ we either have $\rho( b=x )=0$ or $\{b=x \}$ is an atom.
  \item \label{it:b2} for all $s,t\in R$, we have $a(s)>a(t)$ implies that $b(s) > b(t)$.  
    % \item \label{it:b3} if the set $A=\{b=x\}$ has positive measure (and is thus an atom), we can find
    %   $\varepsilon>0$ s.t. $A=\{x-\varepsilon<b<x+\varepsilon\}$.
    %   \textbf{todo: (3) brauchmer wahrscheinlich nimmer}
  \end{enumerate} 
\end{proposition}

\begin{proof}
  Before we begin with the construction of the function $b$ satisfying properties \eqref{it:b1} and \eqref{it:b2}, we sketch its idea.
  \begin{figure}[hbt]
      \centering
      \includegraphics{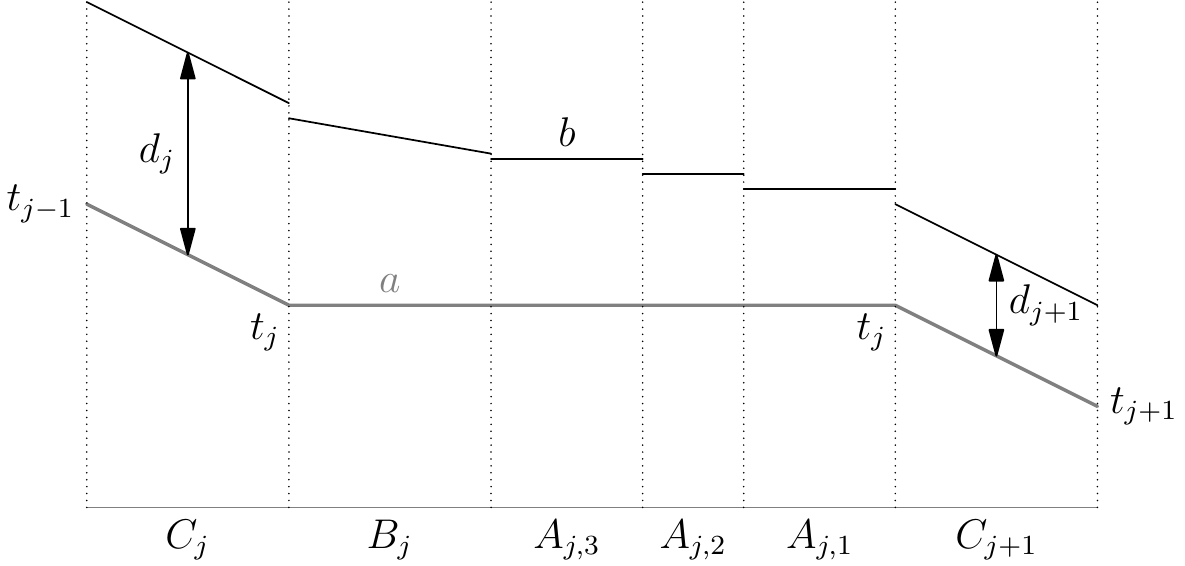}
      \caption{Construction of $b$.}
      \label{fig:1}
    \end{figure}
  
  First, we consider the sets $\{a = t_j\}$ for those $t_j$ such that
  $\rho(a = t_j) > 0$.
  We decompose each of those sets into atoms $A_{j,k}$ of $\mathcal R$ and a continuous part $B_j$.
  We define our function $b$ on $\{a = t_j\}$ in such way that it takes different values on each of
  the atoms $A_{j,k}$. On the continuous part $B_j$, it is defined in such a way that $\rho(\{b = t\}\cap B_j) = 0$ for all $t$, where we use Sierpi\'nski's theorem.

  Let
  $(t_j)_{j\in\mathcal N}$ be the decreasing sequence of all numbers $t$ such that
  \[
  \rho(a=t)>0.
  \]
  Note that there are at most countably many $t$'s with this property, i.e., we can choose
  $\mathcal N = \{1,\dots,N\}$ for some $N\in\N$ or $\mathcal N=\N$.
  Additionally, we set $t_0=\infty$ and define the sets
  \begin{equation*}
    C_j = \{t_j< a <t_{j-1} \},\ j \in \mathcal N
    \qquad\text{and}\qquad
    D = \{a<\inf_{j \in\mathcal N} t_j \}.
  \end{equation*}
  
  First, we specify the function $b$ on the sets $C_j$, $j\in\mathcal N$ and $D$ by
  \begin{equation} \label{eq:offset}
    b(s) = a(s) + d_j,\, s\in C_j
    \qquad\text{and}\qquad
    b(s) = a(s),\, s\in D,
  \end{equation}
  where $d_j = \sum_{\substack{ i\geq j\\ i\in\mathcal N}} 2^{-i}$.
  Note that we have
  \[
  \{a=t_j\} = \bigcup_{k\in\mathcal M_j} A_{j,k} \cup B_j,\qquad j\in\mathcal N,
  \]
  where $\{A_{j,k}\}$ are atoms, $B_j$ is non-atomic and $\mathcal M_j$ is either $\{1,\dots,M_j\}$
  or $\N$.
  The offset $d_j$ introduced in~\eqref{eq:offset} allows us now to define $b$ on $\{a=t_j\}$ such
  that it takes different values on each of the atoms $A_{j,k}$ and such that
  $\rho(\{b = t\}\cap B_j) = 0$ for all $t$.
  Second, we define
  \[
  b(s) = a(s) +d_{j+1} +
  \frac{d_j - d_{j+1}}{2}(1-2^{-k}),\qquad s\in A_{j,k},
  \]
  on each of the atoms $A_{j,k}$.
  In order to define $b$ on the remainder $B_j$ of the set $\{a=t_j\}$ we the invoke Sierpi\'nski's
  Theorem to obtain an increasing function
  $f_j:[0,\rho(B_j)]\to \mathcal R$ such that $\rho(f_j(y))=y$ for all $y\in[0,\rho(B_j)]$ and
  define
  \[
  b(s) = a(s) +d_{j+1} + \frac{d_j - d_{j+1}}{2}
  + \frac{d_j - d_{j+1}}{3\rho(B_j)}\cdot \inf\{\alpha: f_j(\alpha)\ni s\},
  \qquad s\in B_j.
  \]
  Since $b$ satisfies condition~\eqref{it:b2} by construction, it is left to show that it
  satisfies~\eqref{it:b1} as well.
  To this end, let $x \in [0,\infty)$. 
  Note that
  \begin{equation*}
    \{b = x\} = \{a = x - d_j\},
    \qquad \text{whenever $\{b = x\}\cap C_j \neq \emptyset$}.
  \end{equation*}
  Since $t_j < x - d_j < t_{j-1}$, we know $\rho(\{ b = x \}) = 0$. 
  Furthermore, observe that if $\{b = x\} \cap D \neq \emptyset$, then
  $\rho(b=x) = \rho(a=x) = 0$.
  Second, note that if there exist indices $j
  \in \mathcal N$ and $k \in \mathcal M_j$ such that $\{b=x\}\cap A_{j,k} \neq \emptyset$, then $\{b
  = x\} = A_{j,k}$ by construction.  Since $A_{j,k}$ is an atom, so is $\{ b=x \}$.  Third, assume
  there exists an index $j \in \mathcal N$ such that $\{ b = x \}\cap B_j \neq \emptyset$.  Note
  that by construction $\{b = x\} \subset B_j$, thus there exists some number $x'$ such that
  \begin{equation*}
    \{b = x\} = \{z \in B_j\, :\, \inf\{ \alpha\, :\, f_j(\alpha) \ni z \} = x'\}.
  \end{equation*}
  Hence, for all $\varepsilon > 0$ we have
  \begin{equation*}
    \{b = x\} \subset f_j(x'+\varepsilon)\setminus f_j(x'-\varepsilon),
  \end{equation*}
  and, by the properties of $f_j$, we conclude $\rho(b = x) = 0$.
\end{proof}

\begin{lemma}\label{lem:continuous_partial_sums_lemma}
  Let $([0,\alpha),\mathcal B,\kappa)$ be a finite, signed measure space satisfying
  \begin{equation} \label{eq:lem:continuous_partial_sums_lemma}
    \kappa[0,t) \geq 0,
    \qquad 0 \leq t \leq \alpha.
  \end{equation}
  Let $f \in L_1(\kappa)$ be non-negative and decreasing.
  Then
  \begin{equation*}
    \int_{[0,\alpha)} f(t) \dif \kappa(t) \geq 0.
  \end{equation*}
\end{lemma}

\begin{proof}
  First, observe that it is enough to show the assertion of the lemma for simple functions $f$ of the form
  \begin{equation*}
    f = \sum_{j=1}^m f_j \mathbbm 1_{B_j},
  \end{equation*}
  where $f_j \geq 0$ is decreasing and $\{B_j\}$ is a partition of $[0,\alpha)$ of measurable sets
  such that $\sup B_j = \inf B_{j+1}$, $1\leq j \leq m-1$.
  Note that these conditions imply that $B_j$ is a connected subset of $[0,\alpha)$.
  We define $g_j = \kappa(\bigcup_{i=1}^j B_i)$, $1 \leq j \leq m$ and $g_0=0$ and
  observe that~\eqref{eq:lem:continuous_partial_sums_lemma} implies $g_j \geq 0$.
  This is true since $\bigcup_{i=1}^j B_i$ is either $[0,\sup B_j)$ or $[0,\sup B_j]$ and 
  $\kappa[0,t] = \lim_{n} \kappa[0,t+1/n) \geq 0$ for all $0 \leq t < \alpha$.
  Using partial summation we see
  \begin{align*}
    \int_{[0,\alpha)} f(t) \dif \kappa(t)
    & = \sum_{j=1}^m f_j \kappa(B_j)
    = \sum_{j=1}^m f_j (g_j - g_{j-1})\\
    & = f_m g_m - \sum_{j=1}^m (f_j - f_{j-1}) g_{j-1} \geq 0,
  \end{align*}
  since $f\geq 0$ is decreasing and $g_j \geq 0$ as noted before.
\end{proof}

% % % % % % % % % % % % % % % % % % % % % % % % % % % % % %
% % % % % % % % % % % % % % % % % % % % % % % % % % % % % %
\section{Proof of the main theorem}\label{sec: proof main}
% % % % % % % % % % % % % % % % % % % % % % % % % % % % % %
% % % % % % % % % % % % % % % % % % % % % % % % % % % % % %

The purpose of this section is to prove Theorem~\ref{thm:main}. We start with some necessary definitions and lemmata.

% % % % % % % % % % % % % % % % % % % % % % % % % % % % % %
\subsection{Preparatory definitions and results}
% % % % % % % % % % % % % % % % % % % % % % % % % % % % % %

We define the measure space
\begin{equation*}
  (S,\Sigma,\sigma):=(\{1,\dots, n\} \times \Omega,\mathcal P(\{1,\dots,n\})\otimes \mathcal F,\delta\otimes \mu),
\end{equation*}
where $(\Omega,\mathcal F,\mu)$ is the probability space from Theorem~\ref{thm:main}, $\delta$ is the counting measure on $\{1,\dots, n\}$ and $\mathcal P(\{1,\dots,n\})$ denotes the power set of $\{1,\dots,n\}$. Observe that $\sigma(S)=n$ and a measurable subset  $A\subset S$ is an atom in $S$ if and only if $A=\{i\}\times A'$ up to a $\sigma$-null set for some $1\leq i\leq n$ and some atom $A'$ in $\Omega$.

Let $a:S\to \mathbb R$ be a measurable function with respect to Lebesgue measure on $\mathbb{R}$,
which will be fixed throughout the entire section. Note that, without restriction, we assume that
$a$ is non-negative.
We apply Proposition~\ref{pro:construction-b} to the function $a$ and obtain a measurable function
$b:S\to[0,\infty)$ with the following properties:
\begin{enumerate}[(i)]
\item  for all $x\in[0,\infty)$ we either have $\sigma( b=x )=0$ or $\{b=x \}$ is an atom,
\item  for all $s,t\in S$, we have $a(s) > a(t)$ implies that $b(s) > b(t)$.  
\end{enumerate} 
We define the set of all measurable functions on $S$ that are ordered in the same way as $b$ by
\begin{equation} \label{eq:A_b}
  \begin{aligned}
    \mathcal A_b
    := \{d\, :\, S\to &[0,\infty)\text{ measurable\,} \mid \\
    & b(x) \leq b(y)\implies  d(x)\leq d(y) \text{ for all } x,y\in S\}.
  \end{aligned}
\end{equation}
This means that, if $b$ is constant on some set $B$, then any $d\in\mathcal A_b$ is constant on $B$ as well. Note that in general, $d\in\mathcal A_b$ may be constant on some set $B_0$, where $b$ is not.

Moreover, we define the function $h:[0,\infty)\to \Sigma$ by
\[
h(t):=\bigcup_{j=1}^\infty \{b\geq b^*(t-1/j)\}.
\]
Roughly speaking, $h(t)$ is the subset of $S$ having approximately measure $t$, on which $b$ takes
its largest values.
We single out those parameter values $t$ such that $\sigma(h(t))=t$ by setting
\begin{equation*}
  U:=\{t\in [0,n]:\sigma(h(t))=t\}.
\end{equation*}
Since $U$ plays an important role in what follows, we first investigate some of its properties.
\begin{lemma}\label{lem:U-properties}
  The set $U$ has the following properties:
  \begin{enumerate}[(i)]
  \item $n \in U$,\label{item:lem:U-properties-1}
  \item For all $t \in [0,n]$ we have that
    $t \in U^c$ if and only if there exists an open interval $V \ni t$ such that
    $b^*$ is constant on $V$.
    \label{item:lem:U-properties-2}
  \item $U$ is closed.\label{item:lem:U-properties-3}
  \item If $(c,d) \subset U^c$, then $b^*$ is constant on $(c,d)$.
    \label{item:lem:U-properties-4}
  \end{enumerate}
\end{lemma}
\begin{proof}
  \eqref{item:lem:U-properties-1} By definition of $h$, and since $b$ and $b^*$ are equimeasurable, we have for all positive integers $j$:
  \[
  \sigma(h(n))\geq \sigma\big(b\geq b^*(n-1/j)\big) =\lambda\big(b^*\geq b^*(n-1/j)\big)\geq n-1/j,
  \]
  i.e., $\sigma(h(n))\geq n$. On the other hand, $\sigma(h(n))\leq \sigma(S)=n$.

  \eqref{item:lem:U-properties-2} Let $t \in U^c$. Thus, there exists an index $j$ satisfying
  \[
  \sigma(h(t))\geq \sigma\big(b\geq b^*(t-1/j)\big)>t.
  \]
  This implies that $\lambda\big(b^*\geq b^*(t-1/j)\big)>t$ and so there exist two points $t_0,t_1$
  with $t_0<t<t_1$ such that $b^*(t_0)=b^*(t_1)$. But since $b^*$ is decreasing, $b^*$ is constant
  on $(t_0,t_1)$.
  On the other hand, let $b^*$ be constant on the interval $(t-2\varepsilon,t+2\varepsilon)$. This
  implies that $h(t)=h(t+\varepsilon)$ and so,
  \[
  \sigma(h(t))=\sigma(h(t+\varepsilon))\geq t+\varepsilon>t,
  \]
  i.e., $t\in U^c$.

 \eqref{item:lem:U-properties-3} This is an immediate consequence
  of~\eqref{item:lem:U-properties-2}. 

  \eqref{item:lem:U-properties-4} Finally, let $c < d$ with $(c,d) \subset U^c$ and $I \subset (c,d)$ be a compact interval.
  For every $t \in I$ we use~\eqref{item:lem:U-properties-2} to choose an open interval $V(t)$
  containing $t$ on which $b^*$ is constant.
  By compactness, $b^*$ is constant on $I$, and since $I$ was arbitrary,~\eqref{item:lem:U-properties-4} is proved.
\end{proof}

We now recall the assumptions of Theorem~\ref{thm:main}.
The collection $G$ is a subset of all mappings from $\{1,\dots ,n\}\to \Omega$ and $\Pro$ is a probability measure on $G$  satisfying
\begin{enumerate}[(i)]
\item $\Pro(g(i)\in A)=\mu(A)$,
\item $\Pro(g(i_1)\in A_1, g(i_2)\in A_2 ) \leq C_G \mu(A_1)\mu(A_2)$
\end{enumerate}
for all $i\in I$, all different indices $i_1,i_2\in I$, and all measurable sets $A,A_1,A_2\subset \Omega$.
Next, for all $t\in [0,n]$, we define the random variable
\[
X_t:G\to \{0,\dots,n\},\qquad g\mapsto |g\cap h(t)|,
\]
where $|\cdot|$ denotes the cardinality and we interpret $g$ as the graph of $g$, i.e.
$g\cap h(t) = \{(i,g(i))\in h(t)\, :\, i \in I\}$.
Observe that the definition of $h$ and property~\eqref{item:lem:U-properties-4} of the above lemma imply that for $t\in [0,n]:$
\[
X_t = X_{u_0},\qquad \text{where } u_0=\inf\{u\in U:u\geq t\}.
\]
We will now study some properties of these random variables that are essential for the proof of Theorem~\ref{thm:main}.

\begin{proposition}\label{prop:EXt}
  The random variables $(X_t)$ have the following properties:
  \begin{align*}
    \E X_t&=\sigma(h(t)),\qquad  t \in [0,n],\\
    \E X_t^2&\leq t(1+C_Gt),\qquad t \in U.
  \end{align*}
\end{proposition}
\begin{proof}
  Let $t\in [0,n]$. We have $X_t(g)=\sum_{i=1}^n Y_i^t(g),$ with $Y_i^t(g):=|\{(i,g(i))\}\cap h(t)|\in\{0,1\}$	and, since $h(t)$ is a measurable subset of $S$, we can write
  \[
  h(t)=\bigcup_{i=1}^n \{i\}\times A_i^t
  \]
  with some  measurable sets $A_1^t,\dots, A_n^t\subset\Omega$. Therefore, by assumption \eqref{eq:condition 1} in Theorem~\ref{thm:main},
  \[
  \E X_t=\sum_{i=1}^n \E Y_i^t=\sum_{i=1}^n \Pro(g(i)\in A_i^t)=\sum_{i=1}^n \mu(A_i^t) = \sigma(h(t)).
  \]
  Now we assume $t \in U$ and estimate $\E X_t^2$ using assumption \eqref{eq:condition 2} of Theorem~\ref{thm:main}:
  \begin{align*}
    \E X_t^2&=\sum_{i=1}^n \E Y_i^t+\sum_{i\neq j} \E Y_i^t\cdot Y_j^t\\
    &\leq t+\sum_{i\neq j} \Pro(g(i)\in A_i^t, g(j)\in A_j^t) \\
    &\leq t+C_G\Big(\sum_{i=1}^n \mu(A_i^t)\Big)^2=t(1+C_G t),
  \end{align*}
  where we used that, by definition, $\sigma(h(t))=t$ for any $t\in U$.
\end{proof}

\begin{proposition} \label{prop:Xt geq t/2} For all $t\in [0,n]$,
  \[
  \Pro(X_{t}\geq t/2) \geq \frac{t}{4+4C_Gt}.
  \]
\end{proposition}
\begin{proof}
  First, we assume that $t\in U$. Then, Paley-Zygmund's inequality (Theorem~\ref{thm:paley-zygmund})
  in combination with Proposition \ref{prop:EXt} and the choice $\theta=1/2$ imply the desired
  inequality. If $t\in U^c$, define $u_0:=\inf \{u\in U: u\geq t\}$. Hence, property
  \eqref{item:lem:U-properties-4} of Lemma \ref{lem:U-properties} implies $X_t=X_{u_0}$.
  Therefore,
  \begin{align*}
    \Pro(X_t\geq t/2) = \Pro(X_{u_0}\geq t/2) \geq \Pro(X_{u_0}\geq u_0/2)\geq \frac{u_0}{4+4C_G u_0}.
  \end{align*}
  Since $u_0\geq t$ and the function $s\mapsto s/(4+4C_G s)$ is increasing, the result follows.
\end{proof}

\begin{cor}\label{cor:CG min}
  For $t\in [0,n]$, we have
  \[
  \Pro(X_t\geq 1)\geq\min\Big\{\frac{t}{8},\frac{1}{8C_G}\Big\}.
  \]
\end{cor}
\begin{proof}
  If $t\leq 1/C_G$, we obtain from Proposition~\ref{prop:Xt geq t/2}, the fact that $X_t$ takes only integer values, and because $C_G\geq 1$
  \begin{equation}\label{eq:X_t geq 1:1}
  \Pro(X_t\geq 1) = \Pro(X_t \geq t/2) \geq \frac{t}{4+4C_Gt} \geq \frac{t}{8}.
  \end{equation}
  If $t>1/C_G$, we get from Proposition~\ref{prop:Xt geq t/2} and since $X_t$ takes only integer values
  \begin{equation}\label{eq:X_t geq 1:2}
  \Pro(X_t\geq 1)\geq \Pro(X_{1/C_G}\geq 1) = \Pro\Big(X_{1/C_G} \geq \frac{1}{2C_G}\Big) \geq \frac{1}{8C_G}.    
  \end{equation}
  Combining \eqref{eq:X_t geq 1:1} and \eqref{eq:X_t geq 1:2} concludes the proof.
\end{proof}

As a matter of fact, we will use this corollary in the form
\[
  \Pro(X_t\geq 1)\geq\min\Big\{\frac{t}{8},\frac{1}{8C_G}\Big\}\Pro(X_\ell \geq 1),\qquad 1\leq \ell \leq n.
\]

\begin{cor}\label{cor: Xt geq k}
  Let $k\in\N$ with $1\leq k\leq n$. Then, for $t\in [2k,n]$,
  \[
  \Pro(X_t\geq k)\geq \frac{1}{2+4C_G}.
  \]
\end{cor}
\begin{proof}
  This is a direct consequence of Proposition~\ref{prop:Xt geq t/2}.
%  Since $k\leq t/2$, we obtain
%  \[
%  \Pro(X_t\geq k)\geq \Pro(X_t\geq t/2)\geq \frac{t}{4+4C_G t},
%  \]
%  by Proposition~\ref{prop:Xt geq t/2}. The fact that $t\geq 2k\geq 2$ now implies the desired inequality.
\end{proof}

\begin{cor}\label{cor:lower estimate k-max}
  For all $k\in\N$ in the range $1\leq k\leq \ell/2$, 
  \[
  \E \kmax_{1\leq i \leq n} \mathbbm 1_{h(\ell)} \geq \frac{1}{2+4C_G}.
  \]
\end{cor}
\begin{proof}
  Let $k$ be an integer in the range $1\leq k\leq \ell/2$. Using Proposition~\ref{prop:Xt geq t/2} with $t=2k$, we obtain 
  \begin{align*}
    \E \kmax_{1\leq i \leq n} \mathbbm 1_{h(\ell)}
    &\geq \E\big(\kmax_{1\leq i \leq n} \mathbbm 1_{h(\ell)};X_{2k}\geq k\big)\\
%    &\geq \int_{\{g: X_{2k}(g)\geq k\}} \kmax_{1\leq i \leq n} \mathbbm 1_{h(\ell)}(i,g(i)) \dif\Pro(g)\\
    & = \Pro(X_{2k}\geq k)
    \geq \frac{2k}{4+8C_Gk}
    \geq \frac{1}{2+4C_G}. \qedhere
  \end{align*}
\end{proof}

% % % % % % % % % % % % % % % % % % % % % % % % % % % % % %
\subsection{Reduction to Boolean functions}
% % % % % % % % % % % % % % % % % % % % % % % % % % % % % %

In this section, we estimate the expression
$\int_{G}  \sum_{k=1}^{\ell}\kmax\limits_{1\leq i \leq n} |a(i,g(i))| \dif \Pro(g)$ occurring in
Theorem~\ref{thm:main} for general matrices $a$ by the same expression with $a$ replaced
by some averaged matrix $\widetilde a$. In order to begin our investigation, we first have to give a
few definitions. For a measurable function $f \in \mathcal A_b$ we set
\begin{equation} \label{eq:cut-off}
  \widetilde f:= \frac{1}{\ell}\int_0^\ell f^*(s)\dif s \cdot \mathbbm 1_{h(\ell)} \qquad
  \text{and}\qquad f_t:= \mathbbm 1_{h(t)},\, t\in [0,n].
\end{equation}
Observe that both $\widetilde f \in \mathcal A_b$ and $f_t\in \mathcal A_b$.
Moreover, we write
\begin{equation*} %\label{eq:definition-S_k}
  S_k(f)(g) := \kmax_{1\leq i \leq n} f(i,g(i)) \qquad\text{and}\qquad S(f)(g) := \sum_{k=1}^{\ell}S_k(f)(g)
\end{equation*}
for all $f\in \mathcal A_b$ and $g \in G$.
Then, for any $k\in\N$ with $1\leq k \leq \ell$,
\begin{equation}\label{eq:exp_S_k}
  \E S_k(a_t) = \Pro(S_k(a_t)=1) = \Pro (X_t\geq k),\qquad t\in [0,n],
\end{equation}
and, using also the equation in Proposition \ref{prop:EXt},
\begin{equation}\label{eq:exp_S_k_a_tilde}
  \E S_k(\widetilde{a_t}) 
  % = \frac{\min\{\sigma(h(t)), \ell\}}{\ell}\Pro\Big(S_k(\widetilde{a_t}) = \frac{1}{\ell}\int_{0}^\ell a^*_t(s)\dif s\Big)
  = \frac{\min\{\E X_t,\ell\}}{\ell} \Pro(X_\ell \geq k),\qquad t\in[0,n],
\end{equation}
where $\widetilde{a_t} = (a_t)\widetilde\ $, i.e. we first apply the operation $\cdot_t$ and then
the operation $\widetilde \cdot$.

We first establish our result for the special functions $a_t$ in Proposition \ref{prop:a_t_tilde},
which will then allow us to prove (cf. Subsection \ref{subsec:conclusion}) the same result for general functions $a$ in Theorem~\ref{thm:average_a}.

\begin{proposition}\label{prop:a_t_tilde-lower}
  For all $t \in [0,n]$ we have
  \begin{equation*}
    \E S(a_t) \leq (6+12C_G)\cdot \E S(\widetilde{a_t}).
  \end{equation*}
\end{proposition}
\begin{proof}
  First, assume $\ell=1$.
  Then, equation \eqref{eq:exp_S_k_a_tilde} and Corollary~\ref{cor:CG min} yield
  \begin{align*}
    \E S(\widetilde{a_t})
    &\geq \min\{\E X_t,1\}\Pro(X_1\geq 1) \\
    &\geq \frac{1}{8C_G}\min\{\E X_t,1\}
    \geq \frac{1}{8C_G} \Pro(X_t \geq 1)
    =\frac{1}{8C_G} \E S(a_t),
  \end{align*}
  where we used \eqref{eq:exp_S_k} in the latter equality.
  Second, assume that $\ell \geq 2$.
  Due to equation~\eqref{eq:exp_S_k_a_tilde} we have
  \begin{align*}
    \E S(\widetilde{a_t})
    & = \frac{\min\{\E X_t,\ell\}}{\ell} \sum_{k=1}^\ell \Pro(X_\ell \geq k)
     \geq \frac{\min\{\E X_t,\ell\}}{\ell} \sum_{k=1}^{\ell/2} \Pro(X_\ell \geq k).
  \end{align*}
  Then, Corollary~\ref{cor: Xt geq k} and~\eqref{eq:exp_S_k} give us
  \begin{equation*}
    (6+12C_G)\cdot\E S(\widetilde{a_t})
    \geq  \min\{\E X_t, \ell\}
    \geq  \sum_{k=1}^\ell \Pro(X_t \geq k)
    = \E S (a_t),
  \end{equation*}
  where we used that
  \[
  \E X_t = \sum_{k=1}^n \Pro(X_t\geq k).\qedhere
  \]
\end{proof}

\begin{proposition}\label{prop:a_t_tilde}
  For all $t\in U$, we have
  \[
  \E S(\widetilde{a_t})\leq (8+16C_G)\cdot\E S(a_t).
  \]
\end{proposition}
\begin{proof}
  Combining \eqref{eq:exp_S_k} and \eqref{eq:exp_S_k_a_tilde}, we see that it is enough to prove the inequality
  \begin{equation}\label{eq:exp_ineq_to_show}
    \frac{t}{\ell}\sum_{k=1}^\ell \Pro(X_\ell \geq k) \leq (8+16C_G)\cdot \sum_{k=1}^\ell \Pro(X_t\geq k)
  \end{equation}
  for $t\in U\cap [0,\ell]$.

  First, we assume that $t\leq 2$. Then, Corollary~\ref{cor:CG min} implies
  \begin{align*}
    \frac{t}{\ell}\sum_{k=1}^\ell \Pro(X_\ell\geq k)&\leq t\Pro(X_\ell\geq 1) \leq  t\max\Big\{\frac{8}{t},8C_G\Big\} \Pro(X_t\geq 1).
  \end{align*}
  Since $t\leq 2$, we further get
  \begin{equation}\label{eq:case1_exp}
    \frac{t}{\ell}\sum_{k=1}^\ell \Pro(X_\ell\geq k)
    \leq 16C_G \Pro(X_t\geq 1) \leq  16C_G\sum_{k=1}^\ell \Pro(X_t\geq k),
  \end{equation}
  which implies~\eqref{eq:exp_ineq_to_show} for $t\leq 2$.

  Second, we assume $2m\leq t\leq 2(m+1)$ for some $1\leq m\leq \ell/2$. In that case, Corollary~\ref{cor: Xt geq k} yields
  \begin{align*}
    \frac{t}{\ell}\sum_{k=1}^\ell \Pro(X_\ell\geq k) 
    &\leq \frac{t}{m}\sum_{k=1}^m \Pro(X_\ell\geq k) \leq \frac{t}{m}\sum_{k=1}^m (2+4C_G)\Pro(X_t\geq k).
  \end{align*}
  Using the inequality $t\leq 2(m+1)$, we conclude
  \begin{align*}
    \frac{t}{\ell}\sum_{k=1}^\ell \Pro(X_\ell\geq k) 
    &\leq (8+16C_G)\cdot\sum_{k=1}^m \Pro(X_t\geq k) \leq (8+16C_G)\cdot \sum_{k=1}^\ell \Pro(X_t\geq k),
  \end{align*}
  which is~\eqref{eq:exp_ineq_to_show} for $t\geq 2$. Combining the latter with~\eqref{eq:case1_exp}, the proof of the proposition is completed.
\end{proof}

\begin{thm}\label{thm:average_a}
  Let $a:S\to [0,\infty)$ be an arbitrary measurable function on $S$. Then,
  \begin{equation}\label{eq:Sa and Satilde}
  \frac{1}{6 + 12C_G}\cdot\E S(a) \leq \E S(\widetilde a)\leq (8+16C_G)\cdot\E S(a).
  \end{equation}
\end{thm}
\begin{proof}
  Defining the $\mathcal F^n$-measurable
  function $u_k:G\to [0,n]$ by
  \[
  u_k(g):=\inf \{t: X_t(g)\geq k\}= \inf \{t: |h(t)\cap g|\geq k\},
  \]
  we first show that $d(f(u_k(g)))$ is well defined for all $d\in \mathcal A_b$, where  $f(t):=b^{-1}(b^*(t))$, and that this expression actually satisfies
  \begin{equation*} %\label{eq:identity-1}
    S_k(d)(g)=d(f(u_k(g))),\qquad d\in\mathcal A_b.
  \end{equation*}
  For the definition of $\mathcal A_b$ see~\eqref{eq:A_b}.
  Observe that $u_k(g)$ is the unique number such that
  \[
  |h(u_k(g) - \varepsilon)\cap g| < k
  \qquad \text{and}\qquad
  |h(u_k(g) + \varepsilon)\cap g|\geq k
  \]
  for all $\varepsilon>0$ and, additionally,
  \begin{equation*}
    f(u_k(g)) = \bigcap_{\varepsilon>0} h(u_k(g)+\varepsilon)\setminus h(u_k(g)-\varepsilon).
  \end{equation*}
  Hence, there exists an element $(i_0,g(i_0)) \in f(u_k(g))$ satisfying
  \begin{equation} \label{eq:kmax-1}
    b(i_0,g(i_0)) = S_k(b)(g).
  \end{equation}
  Furthermore, observe that the definition of $f$ implies
  \begin{equation} \label{eq:kmax-2}
    b(y) = b^*(u_k(g))
    \qquad\text{for all $y \in f(u_k(g))$}.
  \end{equation}
  A consequence of the definition of $\mathcal A_b$ is that \eqref{eq:kmax-1} and~\eqref{eq:kmax-2}
  imply
  \begin{equation*} %\label{eq:kmax_for_d}
    d(i_0,g(i_0)) = S_k(d)(g) = d(y) \qquad\text{for all $y \in f(u_k(g))$}.
  \end{equation*}
  
  Therefore, \eqref{eq:kmax-2} together with the definition of $\mathcal A_b$ gives us that $d\circ f\,
  :\, u_k(G) \to [0,\infty)$ is a well defined and decreasing function, thus measurable.
  Observe that by changing variables we obtain
  \[
  \E S_k(d)
  = \int_G d(f(u_k(g))) \dif \Pro(g)
  = \int_{u_k(G)} d(f(z)) \dif \Pro_{u_k}(z).
  \]
  We will now show that
  \begin{equation} \label{eq:need-to-show}
    \int_{u_k(G)} d(f(z)) \dif \Pro_{u_k}(z) = \int_{u_k(G)} d^*(z) \dif \Pro_{u_k}(z).
  \end{equation}
  Without loss of generality, we may assume that $d$ is a simple function of the form
  \[
  d=\sum_{j=1}^m d_j \mathbbm 1_{D_j},
  \]
  with $(d_j)$ decreasing and $\{D_j\}$ a disjoint collection of measurable sets.
  Observe that in this case
  \begin{equation} \label{eq:decreasing_rearrangement}
    d^*=\sum_{j=1}^m d_j \mathbbm 1_{[\sum_{i=1}^{j-1} \sigma(D_i),\sum_{i=1}^{j} \sigma(D_i))}
  \end{equation}
  and~\eqref{eq:need-to-show} becomes
  \[
  \sum_{j=1}^m d_j \int_{u_k(G)} \mathbbm 1_{D_j}(f(z))\dif \Pro_{u_k}(z)
  = \sum_{j=1}^m d_j \Pro_{u_k}\bigg( \Big[\sum_{i=1}^{j-1}\sigma(D_i),\sum_{i=1}^{j}\sigma(D_i)\Big)\bigg).
  \]
  Thus, it is sufficient to prove
  \begin{equation} \label{eq:need-to-show-2}
    u_k(G)\cap\{z:\{b=b^*(z)\}\subset D_j\} = u_k(G)\cap\Big[\sum_{i=1}^{j-1}\sigma(D_i),\sum_{i=1}^{j}\sigma(D_i)\Big).
  \end{equation}
  On the one hand, let
  $z \in u_k(G)\cap\Big[\sum_{i=1}^{j-1}\sigma(D_i),\sum_{i=1}^{j}\sigma(D_i)\Big)$, hence,
  $d^*(z) = d_j$.
  Observe that since $d \in \mathcal A_b$ we have
  \begin{equation*} %\label{eq:constant_implies_constant}
    \{y\, :\, b(y) = b^*(z)\} \subset \{y\, :\, d(y) = d^*(z)\} = D_j.
  \end{equation*}
  On the other hand, let $z \in u_k(G)$ be such that $\{y\, :\, b(y)=b^*(z)\} \subset D_j$.
  Note that there exists a unique index $j_0$ such that
  \begin{equation*}
    \{y\, :\, b(y) = b^*(z)\} \subset \{y\, :\, d(y) = d^*(z)\} = D_{j_0}.
  \end{equation*}
  Observe that~\eqref{eq:kmax-1} implies that $u_k(G) \subset \{z\, :\, f(z)\neq \emptyset\}$, hence
  $\{y\, :\, b(y)=b^*(z)\} \neq \emptyset$.
  Thus, the disjointness of the $\{D_j\}$ implies $j=j_0$.
  Therefore, we obtain from~\eqref{eq:decreasing_rearrangement} that
  $z \in \Big[\sum_{i=1}^{j-1}\sigma(D_i),\sum_{i=1}^{j}\sigma(D_i)\Big)$.
  This proves~\eqref{eq:need-to-show-2} and consequently~\eqref{eq:need-to-show}.
  So far we proved that
  \[
  \E S_k(d) = \int_{u_k(G)} d^*(z) \dif \Pro_{u_k}(z),
  \qquad d\in \mathcal A_b.
  \]
  Therefore, setting $\nu=\sum_{k=1}^\ell \Pro_{u_k}$ we obtain
  \[
  \E S(d)=\int_{[0,n]} d^*(z)\dif\nu(z),
  \qquad d\in \mathcal A_b.
  \]
  Recalling~\eqref{eq:cut-off} we observe that
  \[
  \E S(\widetilde a) = \frac{1}{\ell}\int_{[0,\ell)}a^*(s)\dif s\cdot\nu[0,\sigma(h(\ell))).
  \]
  
  Having now introduced the necessary tools for the proof of \eqref{eq:Sa and Satilde}, we first proceed by proving the upper estimate. Observe that with $C:=8+16C_G$, we can write
  \begin{equation} \label{eq:tau-identity}
    C\E S(a)-\E S(\widetilde a)=\int_{[0,n]} a^*(x)\dif \tau(x)
  \end{equation}
  with the signed measure $\dif \tau(x)=C\dif\nu(x)-\frac{\nu[0,\sigma(h(\ell)))}{\ell}\dif \eta(x)$, where $\eta$ is the Lebesgue measure on $[0,\ell)$.
  We have shown in Proposition~\ref{prop:a_t_tilde} that
  \[
  C \nu[0,t)=C\E S(a_t)\geq \E S(\widetilde{a_t})=\frac{\min\{t, \ell\}}{\ell}\nu[0,\sigma(h(\ell))),
  \qquad t \in U,
  \]
  i.e., $\tau[0,t)\geq 0$ for all $t \in U$.
  We will now show that $\tau[0,t) \geq 0$ for all $t \in [0,n]$.
  To this end, let $t \in U^c$.
  Define $u_0 = \inf\{u \geq t\, :\, u \in U\}$ and note that $u_0 \in U$, since $U$ is closed by
  Lemma~\ref{lem:U-properties}.
  Observe that by~\eqref{item:lem:U-properties-2} of Lemma~\ref{lem:U-properties}
  $(t-\varepsilon, u_0) \subset U^c$ for some $\varepsilon > 0$.
  Hence, \eqref{item:lem:U-properties-4} of Lemma~\ref{lem:U-properties} implies that $b^*$ is
  constant on $(t-\varepsilon, u_0)$, which means by definition of $h$ that
  $h(t) = h(u_0)$.
  As a consequence we obtain
  \begin{equation*}
    C\nu[0,t)
    = C \nu[0,u_0)
    \geq \frac{\min\{u_0, \ell\}}{\ell} \nu[0,\sigma(h(\ell)))
    \geq \frac{\min\{t, \ell\}}{\ell} \nu[0,\sigma(h(\ell))),
  \end{equation*}
  i.e. $\tau[0,t) \geq 0$ for all $t\in [0,n]$.
  Applying Lemma~\ref{lem:continuous_partial_sums_lemma} to the right hand side of~\eqref{eq:tau-identity} we obtain
  \begin{equation*}
    C\E S(a)-\E S(\widetilde a) \geq 0,
  \end{equation*}
  which concludes the proof of the upper estimate.
  
  The proof of the lower estimate in \eqref{eq:Sa and Satilde} follows along the same lines by just employing Proposition~\ref{prop:a_t_tilde-lower} instead of
  Proposition~\ref{prop:a_t_tilde} and using an appropriate signed measure different than $\tau$. 
  
\end{proof}

% % % % % % % % % % % % % % % % % % % % % % % % % 
\subsection{Conclusion}\label{subsec:conclusion}
% % % % % % % % % % % % % % % % % % % % % % % % % 

As we have seen, we can reduce the case of general $a$ to multiples of functions only taking values zero and one.  We now use this fact to prove Theorem~\ref{thm:main}.

\begin{proof}[Proof of Theorem~\ref{thm:main}]
  First we prove the lower estimate.
  Observe that Theorem~\ref{thm:average_a} and the definition of $\widetilde{a}$ yield
  \begin{equation} \label{eq:proof_main_thm:1}
  \begin{aligned}
    C_1\cdot \E \sum_{k=1}^\ell \kmax_{1\leq i \leq n} a(i,g(i)) & \geq \E \sum_{k=1}^\ell \kmax_{1\leq i \leq n} \widetilde a(i,g(i)) \\
    & = \frac{1}{\ell}\int_0^\ell a^*(t)\dif t \cdot \Big(\E \sum_{k=1}^\ell \kmax_{1\leq i \leq n} \mathbbm 1_{h(\ell)}\Big),
   \end{aligned} 
  \end{equation}
  where $C_1=8(1+2C_G)$.
  If $\ell = 1$, then Corollary~\ref{cor:CG min} implies
  \begin{align*}
    C_1\E \max_{1\leq i \leq n} a(i,g(i))
    & \geq  \int_0^1 a^*(t)\dif t \cdot \E \max_{1\leq i \leq n} \mathbbm
    1_{h(1)}(i,g(i))\\
    & \geq \int_0^1 a^*(t)\dif t \cdot \Pro(X_1 \geq 1)\\
    & \geq \frac{1}{8C_G} \int_0^1 a^*(t)\dif t.
  \end{align*}
  For $\ell \geq 2$ we use Corollary~\ref{cor:lower estimate k-max} and see
  \begin{align*}
    C_1\cdot \E \sum_{k=1}^\ell \kmax_{1\leq i \leq n} a(i,g(i))
    &\geq  \frac{1}{\ell}\int_0^\ell a^*(t)\dif t \cdot \Big(\E
      \sum_{k=1}^{\ell/2} \kmax_{1\leq i \leq n} \mathbbm 1_{h(\ell)}(i,g(i))\Big)\\
    &\geq \frac{1}{6(1 +2 C_G)} \int_0^\ell a^*(t)\dif t,
  \end{align*}
  which proves the lower estimate.
  
  Now, we proceed with the upper estimate.
  For all $\ell \geq 1$ we have by Theorem~\ref{thm:average_a}
  \begin{align*}
    c_1 \cdot\E \sum_{k=1}^\ell \kmax_{1\leq i \leq n} a(i,g(i))
    & \leq  \E \sum_{k=1}^\ell \kmax_{1\leq i \leq n} \widetilde a(i,g(i))\\
    & =  \frac{1}{\ell} \int_0^\ell a^*(t)\dif t\cdot \E \sum_{k=1}^\ell
    \kmax_{1\leq i \leq n} \mathbbm 1_{h(\ell)}(i,g(i))\\
    & \leq  \int_0^\ell a^*(t)\dif t,
  \end{align*}
  where $c_1=1/(6 + 12 C_G)$. This concludes the proof of the theorem.
\end{proof}

% % % % % % % % % % % % % % % % % % % % % % % % % % % 
% % % % % % % % % % % % % % % % % % % % % % % % % % % 
\section{An application to Orlicz spaces} \label{sec:applications}
% % % % % % % % % % % % % % % % % % % % % % % % % % % 
% % % % % % % % % % % % % % % % % % % % % % % % % % % 

We will present an application of our main result (cf. Theorem~\ref{thm:main}) dealing with averages of order statistics on random sequences. The expressions for the bounds on the expectations that we obtain for \eqref{eq:k decreasing rearrangement random variables} are in terms of Orlicz norms and rather simple (cf. Theorem~\ref{thm:main sequences of random variables}). Note that by our ``direct'' approach, we recover Corollaries $2$ and $3$ from \cite{GLSW1}. 

We will estimate the following expression: 

\begin{equation*} % \label{eq:average order statistics random variables}
  \E \sum_{k=1}^\ell \kmax_{1\leq i \leq n} |x_iX_i|,
\end{equation*}
where $X_1,\dots,X_n$ are independent copies of a random variable $X:(\Omega,\mu) \to \R$ with
$\E|X|< \infty$. Those expressions were already studied in \cite{GLSW1,GLSW2}. There, the argument
in the proof is built upon an estimate involving only the largest order statistic and combinatorial
results that were obtained in \cite{KS1,KS2}.
% \begin{equation}\label{eq:independent average}
%   \frac{1}{n^n} \sum_{j_1,\dots,j_n=1}^n \max_{1\leq i \leq n} |a_{ij_i}| \simeq %\frac{1}{n}\sum_{j=1}^n s(j)
% \end{equation}
% i.e., 
% \eqref{eq:sum k-max} for $\ell=1$.
% and requires a passage from the left hand side of \eqref{eq:independent average} to the left hand side of \eqref{eq:average order statistics random variables}.
However, with Theorem~\ref{thm:main}, problems of this form can be approached directly.

Recall that $M^*$ in Theorem~\ref{thm:main sequences of random variables} is defined by 
\begin{equation*}
  M^*\Big( \int_0^\beta X^*(y) \dif y\Big) = \frac{\beta}{\ell}, \qquad 0 \leq \beta \leq 1.
\end{equation*}
The following Lemma is a continuous version of Lemma 2.1 in \cite{KS1} and our proof follows along the same lines. 

\begin{lemma}\label{lem:extreme points balls}
  Let
  \begin{equation} \label{eq:convex hull}
    B=\con \bigg\{ \Big(\varepsilon_i \int_0^{\alpha_i}X^*(y) \dif y \Big)_{i=1}^n : \varepsilon_i = \pm 1 , \sum_{i=1}^n \alpha_i = \ell \bigg\}. 
  \end{equation}
  Then we have
  \[
  B \subset B_{M^*} \subset 3 B.
  \]
\end{lemma}
\begin{proof}
  First, we show the left inclusion. Let $z\in B$. Then
  \[
  \sum_{i=1}^n M^*(|z_i|) = \sum_{i=1}^n M^* \Big(\int_0^{\alpha_i}X^*(y) \dif y \Big) = \sum_{i=1}^n \alpha_i/\ell= 1.
  \]
  To show the other inclusion, let $z_1 \geq \dots \geq z_n>0$ and 
  \[
  \sum_{i=1}^n M^*(z_i) = 1.
  \]
  We write $z = z'+z'' = (z_1,\dots,z_r,0\dots,0)+ (0,\dots,0,z_{r+1},\dots,z_n)$, where $r$ is chosen such that $M^*(z_i)>1/n$ for all $1\leq i \leq r$, and $M^*(z_i)\leq 1/n$ for all $r\geq i+1$.
  
  We have
  \[
  M^* \Big( \int_{0}^{\ell/n} X^*(y) \dif y \Big) = \frac{1}{n}.
  \]
  Therefore, $z'' \leq \big(\int_{0}^{\ell/n} X^*(y) \dif y, \dots, \int_{0}^{\ell/n} X^*(y) \dif y \big)=:w\in\R^n$. Since $w \in B$, we also have $z''\in B$.
  
  It is now left to show that $z'\in B$. There exist indices $k_i\geq 1$ for $1\leq i \leq r$ such that
  \begin{equation} \label{eq: M^* estimate k_i}
    \frac{k_i}{n} \leq M^*(z_i) \leq \frac{k_i+1}{n}.
  \end{equation}
  Since
  \[
  \sum_{i=1}^r \frac{k_i}{n} = \sum_{i=1}^r M^*\Big(\int_0^{\ell k_i/n}X^*(y)\dif y \Big) \leq \sum_{i=1}^r M^*(z_i) \leq 1,
  \]
  and $\sum_{i=1}^r \ell k_i/n \leq \ell$, we immediately obtain
  \[
  \Big( \int_0^{\ell k_1/n}X^*(y)\dif y,\dots, \int_0^{\ell k_r/n}X^*(y)\dif y,0\dots,0\Big) \in B.
  \]
  Using \eqref{eq: M^* estimate k_i}, we see that
  \begin{align*}
    2z & = \Big( 2\int_0^{\ell k_1/n}X^*(y)\dif y,\dots, 2\int_0^{\ell k_r/n}X^*(y)\dif y,0\dots,0\Big) \\
    & \geq \Big( \int_0^{2\ell k_1/n}X^*(y)\dif y,\dots, \int_0^{2\ell k_r/n}X^*(y)\dif y,0\dots,0\Big) \geq z'.
  \end{align*}
  Thus $z' \in 2B$. We conclude that $z\in 3B$.
\end{proof}

\begin{proof}[Proof of Theorem~\ref{thm:main sequences of random variables}]
  We will assume that the independence of $X_1,\dots, X_n$ is realized through $n$ factors with $G=\Omega^n$ and $X_i(g)=X(g(i))$ for $g\in G$ and a canonical random variable $X$ with the same distribution as $X_1,\dots,X_n$. This first means that
  \[
  \E \sum_{k=1}^\ell \kmax_{1\leq i \leq n} |x_iX_i| = \E_{G} \sum_{k=1}^\ell \kmax_{1\leq i \leq n}|x_iX(g(i))|.
  \]
  Defining $a:\{1,\dots,n\}\times \Omega \to \R$ by 
  \[
  a(i,\omega):= x_i X(\omega),
  \]
  and setting
  $\Pro=\bigotimes_{i=1}^n \mu,$
  we obtain by Theorem~\ref{thm:main}
  \begin{equation*} %\label{eq:average order statistic decreasing ra}
    \E \sum_{k=1}^\ell \kmax_{1\leq i \leq n} |x_iX_i| = \int_G \sum_{k=1}^\ell \kmax_{1\leq i \leq n} |a(i,g(i))| \dif \Pro(g)\simeq \int_0^\ell a^*(t)\dif t.
  \end{equation*}
  Observe that we also have 
  \begin{equation*} %\label{eq:decreasing and sup}
    \int_0^\ell a^*(t)\dif t = \sup_{\sum \alpha_i=\ell} \sum_{i=1}^n x_i \int_0^{\alpha_i} X^*(t)\dif t,
  \end{equation*}
  by approximation of $X$ with simple functions. Therefore, 
  \[
  \E \sum_{k=1}^\ell \kmax_{1\leq i \leq n} |x_iX_i| \simeq \sup_{\sum \alpha_i=\ell} \sum_{i=1}^n x_i \int_0^{\alpha_i} X^*(t)\dif t.
  \]
  With $B$ as in \eqref{eq:convex hull}, using Lemma \ref{lem:extreme points balls} and~\eqref{eq:dualOrlicz}, we further obtain 
  \[
  \E \sum_{k=1}^\ell \kmax_{1\leq i \leq n} |x_iX_i| \simeq \sup_{y\in \ext B} \sum_{i=1}^n x_iy_i \simeq \sup_{y\in B_{M^*}} \sum_{i=1}^n x_iy_i \simeq \|x\|_{M}.
  \]
  This concludes the proof.
\end{proof}

% % % % % % % % % % % % % % % % % % % % % % % % % % % 
% % % % % % % % % % % % % % % % % % % % % % % % % % % 
\section{An application to the Local Theory of Banach spaces} \label{sec:minimal set of maps}
% % % % % % % % % % % % % % % % % % % % % % % % % % % 
% % % % % % % % % % % % % % % % % % % % % % % % % % % 

In this last section we present an example of a set of maps with a minimal number of elements satisfying conditions \eqref{eq:condition 1 cor} and \eqref{eq:condition 2 cor} in Corollary \ref{thm:main LPP}. This is then used to embed certain Orlicz sequence spaces $\ell_M^n$ into $\ell_1^{cn^3}$ using the ``standard'' embedding, which usually provides an embedding into $\ell_1^{n!2^n}$. 

Recall that given two normed spaces $X,Y$ and some constant $C\geq 1$, we say that $X$ $C$-embeds into $Y$ and write $X\stackrel{C}{\hookrightarrow} Y$ if there exists a one to one linear operator $\Psi:X\to \Psi(X)\subseteq Y$ such that $\|\Psi\|\cdot\|\Psi^{-1}\|\leq C$.
Given two isomorphic Banach spaces $X$ and $Y$ the Banach-Mazur distance of $X$ and $Y$ is defined as
    \[
      d_{\textrm{BM}}(X,Y) = \inf\left\{ \|T\|\|T^{-1}\| \,:\, T\in L(X,Y) ~ \hbox{isomorphism} \right\}.
    \]

Before we continue, let us give some historical remarks. The problem, given an $n$-dimensional subspace $X$ of $L_1([0,1],dx)$ and $\varepsilon>0$, what is the smallest $N=N(X,\varepsilon)$ such that there is a subspace $Y$ of $\ell_1^N$ with $d_{\textrm{BM}}(X,Y) \leq 1+\varepsilon$, is extensively studied in the literature. A first breakthrough was made by G. Schechtman in \cite{Sch87}, proving that
\[
N \leq \frac{C}{\varepsilon^2}\log(\varepsilon^{-1})\cdot n^2,
\]
$C>0$ being an absolute constant. 

Based on his work, in \cite{BoLM2} the bound in the dimension was improved to
\[
N \leq \frac{C}{\varepsilon^2}\log(n\varepsilon^{-1})(\log n)^2\cdot n
\]
Later, M. Talagrand in \cite{T90} proved that for any $0< \varepsilon \leq \varepsilon_0$
\begin{equation}\label{eq:talagrand}
N \leq C K(X)^2\varepsilon^{-2}\cdot n,
\end{equation}
where $K(X)$ denotes the $K$-convexity constant of $X$ (cf. \cite{MiSc}). Recall that G. Pisier proved in \cite{Pi80} that if $X\subseteq L_1([0,1],dx)$ and $\dim(X)=n$, then $K(X) \leq C \sqrt{\log n}$. Thus \eqref{eq:talagrand} improves on previous results and gives
\begin{equation}\label{eq:talagrand 2}
N \leq C\varepsilon^{-2}\log(n)n.
\end{equation}
For further information we also refer to
the work of W. B. Johnson and G. Schechtman \cite{JS}, A. Naor and A. Zvavitch \cite{NZ}, J. Bernu\'es and M. L\'opez-Valdes, O. Friedland and O. Gu\'edon, as well as the references therein.

Now, recall that in \cite{BrDa}, using the theorem of de Finetti, J. Bretagnolle and D. Dacunha-Castelle proved that an Orlicz space $\ell_M$ is isomorphic to a subspace of $L_1$ if and only if $M$ is equivalent to a $2$-concave Orlicz function. The corresponding finite-dimensional result was proved by S. Kwapie\'n and C. Sch\"utt in \cite{KS1,S2}. In combination with the previous results, this shows that $\ell_M^n$ $C$-embeds into $\ell_1^{c\log(n)n}$.

The proofs of the results mentioned above are very involved and quite technical.
The embedding we present here is specific and rather simple, with an offset in the dimension $N$.

% % % % % % % % % % % % % % % % % % % % % % % % % % % % % %
\subsection{A minimal set of maps}
% % % % % % % % % % % % % % % % % % % % % % % % % % % % % %

Let $n$ be a power of a prime number and let $\mathbb F_n$ denote the field with $n$ elements.
We define $I_0 = \Omega_0 = \mathbb F_n$ and denote by $\mu_0$ the probability measure on $\Omega_0$ defined
by $\mu_0(\{i\}) = \frac{1}{n}$, for all $i \in \mathbb F_n$.
If we set $G_0 = \{g_{\ell m}\, :\, \ell,m \in \mathbb F_n\}$, where $g_{\ell m}(i) = \ell i +m$ and multiplication and addition is performed in $\mathbb F_n$, then the
probability measure $\Pro_0$ on $G_0$ given by $\Pro_0(\{g\}) = \frac{1}{n^2}$, for all
$g\in G_0$ satisfies conditions~\eqref{eq:condition 1} and~\eqref{eq:condition 2} of
Theorem~\ref{thm:main} with $C_{G_0} = 1$, that is
\begin{enumerate}[(i)]
\item $\Pro_0(g(i) = j) = \frac{1}{n}$, for all $(i,j)\in I_0\times \Omega_0$,
\item $\Pro_0(g(i_1) = j_1, g(i_2) = j_2 ) \leq \frac{1}{n^2}$, for all $(i_1,j_1)\neq(i_2,j_2)$ in $I_0\times \Omega_0$.
\end{enumerate}
We want to point out that the above $G_0$ consists of $n^2$ elements.

We first prove condition (i) and let $(i,j)\in I_0\times \Omega_0$. Then, for an arbitrary $\ell\in \mathbb F_n$, there exists exactly one $m\in \mathbb F_n$, which is given by $m=j-\ell i$ such that $g_{\ell m}(i)=\ell i+m=j$. Therefore, condition (i) is satisfied.
For condition (ii) we note that for all different tuples $(i_1,j_1),(i_2,j_2)\in I_0\times \Omega_0$, in order to have $g_{\ell m}(i_1)=\ell i_1+m=j_1$ and $g_{\ell m}(i_2)=\ell i_2+m=j_2$ for some $\ell,m\in\mathbb F_n$, it is necessary that $i_1\neq j_2$ and in this case $\ell$ is given (uniquely) by $\ell=(j_1-j_2)(i_1-i_2)^{-1}$ and $m=j_1-\ell i_1=j_2-\ell i_2$. Therefore, the event $\{g\in G_0 : g(i_1)=j_1,g(i_2)=j_2\}$ consists of at most one element and the definition of $\mathbb P_0$ implies condition (ii).

In general, we have the following result.
Let $n\in \mathbb N$, define $I_1 = \Omega_1 = \{1,\ldots,n\}$ and set
$\mu_1(\{i\}) = \frac{1}{n}$, for all $1\leq i \leq n$.
If $G_1$ and $\mathbb P_1$ satisfy conditions~\eqref{eq:condition 1} and~\eqref{eq:condition 2} of
Theorem~\ref{thm:main} for some constant $C_{G_1} \geq 1$, then $G_1$ consists of at least
$\frac{n^2}{C_{G_1}}$ elements.

Indeed, assume that conditions (i) and (ii) are satisfied with some constant $C_{G_1}\geq 1$ and assume that $G_1$ consists of less than $n^2/C_{G_1}$ elements. Then, there exists at least one element $g_1\in G_1$ such that $\Pro(\{g_1\}) > C_{G_1} / n^2$. Since, for the choice $j_1=g_1(i_1), j_2=g_1(i_2)$ and arbitrary different $i_1,i_2\in \mathbb F_n$, this $g$ is an element of the event $\{g\in G_1 : g(i_1)=j_1,g(i_2)=j_2\}$. Therefore, by (ii), we get the contradiction
\[
\frac{C_{G_1}}{n^2} < \Pro(\{g_1\}) \leq \Pro(\{g\in G_1 : g(i_1)=j_1,g(i_2)=j_2\}) \leq \frac{C_{G_1}}{n^2}
\]

This shows that, up to a constant factor, the set of functions $G_0$ has the least number of
elements satisfying conditions~\eqref{eq:condition 1} and~\eqref{eq:condition 2} of
Theorem~\ref{thm:main}.

% % % % % % % % % % % % % % % % % % % % % % % % % % % % % %
\subsection{Embedding $\ell_M^n$ into $\ell_1^{cn^3}$}\label{subsec:local theory}
% % % % % % % % % % % % % % % % % % % % % % % % % % % % % %

As an application to Banach space theory, we will now apply Theorem~\ref{thm:main} to $I_0$, $\Omega_0$, $\mu_0$, $G_0$ and $\Pro_0$, defined as above, and prove Theorem \ref{thm:application}. We start by explaining the rough idea before going through the details. Let $M$ be a strictly convex, twice differentiable Orlicz function that is strictly $2$-concave.
We will show that the Orlicz sequence space $\ell_M^n$ $C$-embeds into $\ell_1^{cn^3}$, where $c$ and $C$ are absolute constants independent of $n$ and $M$. This should be compared with the ``standard'' embedding of $\ell_M^n$ into $\ell_1^{n!2^n}$. Recall that the standard embedding (cf. \cite{S2}) is given by 
\[
\Psi_n:\ell_M^n \to \ell_1^{n!2^n},\qquad x \mapsto \frac{1}{n!2^n} \left(\sum_{i=1}^n \varepsilon_i a_{\pi(i)}x_i\right)_{\pi,\varepsilon},
\]
where $a=a(M)\in\R^n$ is chosen in such a way that it generates the Orlicz norm, i.e., 
\[
\frac{1}{n!}\sum_{\pi\in\mathfrak S_n}\left(\sum_{i=1}^n |x_ia_{\pi(i)}|^2 \right)^{1/2} \simeq \|x\|_M.
\]
Indeed, using Khintchine's inequality we then obtain that
\begin{align*}
\|\Psi_n(x)\|_1 & = \frac{1}{n!2^n}\sum_{\pi,\varepsilon} \Big| \sum_{i=1}^n \varepsilon_i a_{\pi(i)}x_i \Big| \\
& \simeq \frac{1}{n!}\sum_{\pi\in\mathfrak S_n}\left(\sum_{i=1}^n |x_ia_{\pi(i)}|^2 \right)^{1/2} \\
& \simeq \|x\|_M.
\end{align*}

So the standard embedding combines Khintchine's inequality with an average over the symmetric group $\mathfrak S_n$, which explains the dimension $n!2^n$ (see also \cite{S1,PS,P,P2} for embeddings of other types of spaces into $L_1$). Instead of taking an average over the whole symmetric group, we rather use our minimal set of maps $G_0$, which has cardinality $n^2$ only, thus obtaining an embedding into $\ell_1^{n^22^n}$. To further decrease the dimension, we will then make use of a result due to J. Bourgain, J. Lindenstrauss, and V.D. Milman that allows to use only $cn$ sign vectors instead of $2^n$. 

Now let us be more precise. To find the sequence $(a_i)_{i=1}^n$ of scalars that generates the Orlicz norm, we use the following result due to C. Sch\"utt (cf. \cite[Theorem 2]{S2}): if $M$ is a strictly convex, twice differentiable Orlicz function that is strictly $2$-concave and satisfies $M^*(1)=1$, then there exists a sequence $a_1,\dots,a_n$ of scalars such that for all $x\in\R^n$,
\begin{equation}\label{eq:schuett average}
\frac{1}{c}\|x\|_M \leq \frac{1}{n!}\sum_{\pi\mathfrak S_n} \Big( \sum_{i=1}^n |x_ia_{\pi(i)}|^2 \Big)^{1/2} \leq c \|x\|_M,
\end{equation}
where $c$ is a constant that does not depend on $n$ and $M$. As a matter of fact, in \cite[Theorem 2]{S2} an explicit formula for the choice of $a$ is given.

J. Bourgain, J. Lindenstrauss, and V. D. Milman \cite{BLM} proved the following: if $v_1,\dots,v_n$ are unit vectors in a normed space $(\R^n,\|\cdot\|)$, then, for all $\delta>0$, there exists a constant $C(\delta)>0$ and $N=C(\delta)n$ sign vectors $\varepsilon^1,\dots,\varepsilon^N\in\{-1,1\}^n$ such that for all $x\in\R^n$
\begin{equation}\label{eq:BLM}
(1-\delta)\ave_{\pm}\Big\|\sum_{i=1}^n \pm x_iv_i \Big\| \leq \frac{1}{N}\sum_{j=1}^N \Big\| \sum_{i=1}^n \varepsilon^j_i x_iv_i \Big\| \leq (1+\delta)\ave_{\pm}\Big\|\sum_{i=1}^n \pm x_iv_i \Big\|.
\end{equation}
For our purpose, it is enough to know that in the setting $v_1=v_2=\dots=v_n=e_1$, where $e_1$ is the first standard unit vector of $\R^n$, $\|\cdot\|=\|\cdot \|_1$, and say $\delta=1/4$, there exists a choice of $N$ sign vectors that satisfy \eqref{eq:BLM}.

%\begin{proposition}
%Let $\|\cdot\|$ be a norm on $\R^n$ and $e_1,\dots,e_n$ be the standard unit vectors. Then for every $0< \delta < \frac{1}{2}$ there are $N$ choices of sign vectors $\varepsilon^j = (\varepsilon^j_1\dots,\varepsilon^j_n)$, $j=1,\dots,N$ with $N\leq cn\varepsilon^{-2}\log(\varepsilon^{-1})$ such that for all $x\in\R^n$
%\[
%(1-\delta) \ave_{\pm}\Big\|\sum_{i=1}^n \pm x_ie_i\Big\| \leq \frac{1}{N}\sum_{j=1}^N \Big\|\sum_{i=1}^n \varepsilon_i^j x_i e_i\Big\| \leq (1+\delta) \ave_{\pm}\Big\|\sum_{i=1}^n \pm x_ie_i\Big\|.
%\]
%\end{proposition}

The last ingredient is a special case of a result we recently obtained in \cite[Theorem 1.4]{LPP}) and reads as follows:  let $n\in\N$, $a\in \R^{n\times n}$, and $1\leq p < \infty$. Let $G$ be a collection of maps from $I=\{1,\dots,n\}$ to $I$ and $C_G>0$ be a constant only depending on $G$. Assume that for all $i,j\in I$ and all different pairs $(i_1,j_1),(i_2,j_2)\in I\times I$ 
\begin{enumerate}[(i)]
\item $\Pro(\{g\in G: g(i)=j\})=1/n$,
\item $\Pro(\{g\in G: g(i_1)=j_1, g(i_2)=j_2 \}) \leq C_G/n^2$.
\end{enumerate}
Then
\begin{align*}
C\bigg[ \frac{1}{n} \sum_{k=1}^n s(k) + \Big(\frac{1}{n}\sum_{k=n+1}^{n^2} s(k)^p\Big)^{1/p}\bigg] & \leq \E \Big( \sum_{i=1}^n |a_{ig(i)}|^p \Big)^{1/p} \\
& \leq \frac{1}{n} \sum_{k=1}^n s(k) + \Big(\frac{1}{n}\sum_{k=n+1}^{n^2} s(k)^p\Big)^{1/p},
\end{align*}
where $(s(k))_{k=1}^{n^2}$ is the decreasing rearrangement of $\{|a_{ij}|:i,j=1,\dots,n \}$ and $C>0$ is a constant only depending on $C_G$.

Note that the conditions in the theorem are satisfied for $G=G_0$, as was shown above, and for $G=\mathfrak S_n$ (cf. \cite[Example 1.2]{LPP}), although with different, but still absolute constants. This means that, especially for $p=2$,
\begin{equation}\label{eq:equivalence averages}
\frac{1}{n^2}\sum_{g\in G_0} \Big( \sum_{i=1}^n |a_{ig(i)}|^2 \Big)^{1/2} \simeq \frac{1}{n!}\sum_{\pi\in\mathfrak S_n} \Big( \sum_{i=1}^n |a_{i\pi(i)}|^2 \Big)^{1/2}
\end{equation}
for all $a\in\R^{n\times n}$.

Let us now prove the embedding result. 
\begin{proof}[Proof of Theorem \ref{thm:application}]
Let $G_0$ be our minimal set of maps. We define the isomorphism $\Psi_n$ by 
\[
\Psi_n:\ell_M^n \to \ell_1^{Cn^3},\qquad x\mapsto \frac{1}{Cn^3}\left(\sum_{i=1}^n \varepsilon^j_i a_{g(i)}x_i\right)_{g\in G_0,j=1,\dots,Cn}.
\]
Then a direct computation as shown in the standard embedding, now using equations \eqref{eq:schuett average}, \eqref{eq:BLM} (in the setting mentioned above), and \eqref{eq:equivalence averages}, shows that
\[
\| \Psi_n(x)\|_1 \simeq \|x\|_M.
\]
This means that there exist absolute constants $C,C_1>0$ such that for all $n\in\N$, $\ell_M^n$ $C_1$-embeds into $\ell_1^{Cn^3}$, where $C,C_1>0$ are independent of $M$.
\end{proof}

\subsection*{Acknowledgments}
R. Lechner is supported by the Austrian Science Fund, FWF P23987 and FWF P22549.
M. Passenbrunner is supported by the Austrian Science Fund, FWF P27723.
J. Prochno is supported by the Austrian Science Fund, FWFM 1628000.

\bibliographystyle{abbrv}
\bibliography{order_statistics_continuous_version}

\def\polhk#1{\setbox0=\hbox{#1}{\ooalign{\hidewidth
  \lower1.5ex\hbox{`}\hidewidth\crcr\unhbox0}}} \def\cprime{$'$}
  \def\polhk#1{\setbox0=\hbox{#1}{\ooalign{\hidewidth
  \lower1.5ex\hbox{`}\hidewidth\crcr\unhbox0}}}
\begin{thebibliography}{10}

\bibitem{BLM}
J.~Bourgain, J.~Lindenstrauss, and V.~D. Milman.
\newblock Minkowski sums and symmetrizations.
\newblock In {\em Geometric aspects of functional analysis (1986/87)}, volume
  1317 of {\em Lecture Notes in Math.}, pages 44--66. Springer, Berlin, 1988.

\bibitem{BoLM2}
J.~Bourgain, J.~Lindenstrauss, and V.~D. Milman.
\newblock Approximation of zonoids by zonotopes.
\newblock {\em Acta Mathematica}, 162(1):73--141, 1989.

\bibitem{BrDa}
J.~Bretagnolle and D.~Dacunha-Castelle.
\newblock Application de l’étude de certaines formes linéaires aléatoires
  au plongement d’espaces de banach dans des espaces $l^p$.
\newblock {\em Annales scientifiques de l'École Normale Supérieure},
  2(4):437--480, 1969.

\bibitem{CRT}
E.~J. Candes, J.~Romberg, and T.~Tao.
\newblock Robust uncertainty principles: Exact signal reconstruction from
  highly incomplete frequency information.
\newblock {\em IEEE Trans. Inf. Theor.}, 52(2):489--509, Feb. 2006.

\bibitem{DN}
H.~A. David and H.~N. Nagaraja.
\newblock {\em Order statistics}.
\newblock Wiley Series in Probability and Statistics. Wiley-Interscience [John
  Wiley \& Sons], Hoboken, NJ, third edition, 2003.

\bibitem{GLMP}
Y.~{Gordon}, A.~{Litvak}, S.~{Mendelson}, and A.~{Pajor}.
\newblock {Gaussian averages of interpolated bodies and applications to
  approximate reconstruction.}
\newblock {\em {J. Approx. Theory}}, 149(1):59--73, 2007.

\bibitem{GLSW1}
Y.~Gordon, A.~Litvak, C.~Sch{\"u}tt, and E.~Werner.
\newblock {\vspace{0cm}}{O}rlicz norms of sequences of random variables.
\newblock {\em Ann. Probab.}, 30(4):1833--1853, 2002.

\bibitem{GLSW2}
Y.~Gordon, A.~Litvak, C.~Sch{\"u}tt, and E.~Werner.
\newblock Geometry of spaces between polytopes and related zonotopes.
\newblock {\em Bull. Sci. Math.}, 126(9):733--762, 2002.

\bibitem{GLSW3}
Y.~Gordon, A.~Litvak, C.~Sch{\"u}tt, and E.~Werner.
\newblock Minima of sequences of {G}aussian random variables.
\newblock {\em C. R. Math. Acad. Sci. Paris}, 340(6):445--448, 2005.

\bibitem{GLSW4}
Y.~Gordon, A.~E. Litvak, C.~Sch{\"u}tt, and E.~Werner.
\newblock On the minimum of several random variables.
\newblock {\em Proc. Amer. Math. Soc.}, 134(12):3665--3675 (electronic), 2006.

\bibitem{GLSW5}
Y.~Gordon, A.~E. Litvak, C.~Sch{\"u}tt, and E.~Werner.
\newblock Uniform estimates for order statistics and {O}rlicz functions.
\newblock {\em Positivity}, 16(1):1--28, 2012.

\bibitem{JS}
W.~B. Johnson and G.~Schechtman.
\newblock Very tight embeddings of subspaces of {$L_p$}, {$1\leq p<2$}, into
  {$l^n_p$}.
\newblock {\em Geom. Funct. Anal.}, 13(4):845--851, 2003.

\bibitem{KS1}
S.~Kwapie{\'n} and C.~Sch{\"u}tt.
\newblock Some combinatorial and probabilistic inequalities and their
  application to {B}anach space theory.
\newblock {\em Studia Math.}, 82(1):91--106, 1985.

\bibitem{KS2}
S.~Kwapie{\'n} and C.~Sch{\"u}tt.
\newblock Some combinatorial and probabilistic inequalities and their
  application to {B}anach space theory. {II}.
\newblock {\em Studia Math.}, 95(2):141--154, 1989.

\bibitem{CCR}
L.~P.~C. Landon P.~Cox, M.~Castro, and A.~Rowstron.
\newblock Pos: A practical order statistics service for wireless sensor
  networks.
\newblock In {\em Proceedings of the 26th IEEE International Conference on
  Distributed Computing Systems}, ICDCS '06, pages 52--64, Washington, DC, USA,
  2006. IEEE Computer Society.

\bibitem{LPP}
R.~Lechner, M.~Passenbrunner, and J.~Prochno.
\newblock Uniform estimates for averages of order statistics of matrices.
\newblock {\em Electron. Commun. Probab.}, 20(27):1--12, 2015.

\bibitem{LZ}
G.~Leus and T.~Zhi.
\newblock Recovering second-order statistics from compressive measurements.
\newblock {\em Proc. IEEE International Workshop on Computational Advances in
  Multi-Sensor Adaptive Processing}, pages 337--340, 2011.

\bibitem{LT77}
J.~Lindenstrauss and L.~Tzafriri.
\newblock {\em Classical {B}anach spaces. {I}}.
\newblock Springer-Verlag, Berlin-New York, 1977.
\newblock Sequence spaces, Ergebnisse der Mathematik und ihrer Grenzgebiete,
  Vol. 92.

\bibitem{MP}
V.~A. Mar{\v{c}}enko and L.~A. Pastur.
\newblock Distribution of eigenvalues in certain sets of random matrices.
\newblock {\em Mat. Sb. (N.S.)}, 72 (114):507--536, 1967.

\bibitem{MiSc}
V.~D. Milman and G.~Schechtman.
\newblock {\em {Asymptotic theory of finite dimensional normed spaces}}.
\newblock Springer-Verlag New York, Inc., New York, NY, USA, 1986.

\bibitem{NZ}
A.~Naor and A.~Zvavitch.
\newblock Isomorphic embedding of {$l^n_p$}, {$1<p<2$}, into
  {$l^{(1+\epsilon)n}_1$}.
\newblock {\em Israel J. Math.}, 122:371--380, 2001.

\bibitem{Pi80}
G.~Pisier.
\newblock Un th\'eor\`eme sur les op\'erateurs lin\'eaires entre espaces de
  {B}anach qui se factorisent par un espace de {H}ilbert.
\newblock {\em Ann. Sci. \'Ecole Norm. Sup. (4)}, 13(1):23--43, 1980.

\bibitem{P}
J.~Prochno.
\newblock A combinatorial approach to {M}usielak-{O}rlicz spaces.
\newblock {\em Banach J. Math. Anal.}, 7(1):132--141, 2013.

\bibitem{P2}
J.~Prochno.
\newblock Musielak--{O}rlicz spaces that are isomorphic to subspaces of $l_1$.
\newblock {\em Ann. Funct. Anal.}, 6(1):84--94, 2015.

\bibitem{PS}
J.~Prochno and C.~Sch{\"u}tt.
\newblock Combinatorial inequalities and subspaces of {$L_1$}.
\newblock {\em Studia Math.}, 211(1):21--39, 2012.

\bibitem{RR}
M.~M. Rao and Z.~D. Ren.
\newblock {\em Theory of {O}rlicz spaces}, volume 146 of {\em Monographs and
  Textbooks in Pure and Applied Mathematics}.
\newblock Marcel Dekker, Inc., New York, 1991.

\bibitem{R}
M.~Rudelson.
\newblock Lower estimates for the singular values of random matrices.
\newblock {\em C. R. Math. Acad. Sci. Paris}, 342(4):247--252, 2006.

\bibitem{Sch87}
G.~Schechtman.
\newblock More on embedding subspaces of {$L_p$} in {$l^n_r$}.
\newblock {\em Compositio Math.}, 61(2):159--169, 1987.

\bibitem{S1}
C.~Sch{\"u}tt.
\newblock Lorentz spaces that are isomorphic to subspaces of {$L^1$}.
\newblock {\em Trans. Amer. Math. Soc.}, 314(2):583--595, 1989.

\bibitem{S2}
C.~Sch{\"u}tt.
\newblock On the embedding of {$2$}-concave {O}rlicz spaces into {$L^1$}.
\newblock {\em Studia Math.}, 113(1):73--80, 1995.

\bibitem{T90}
M.~Talagrand.
\newblock Embedding subspaces of {$L_1$} into {$l^N_1$}.
\newblock {\em Proc. Amer. Math. Soc.}, 108(2):363--369, 1990.

\bibitem{ZLYKZY}
Y.~Zhang, X.~Lin, Y.~Yuan, M.~Kitsuregawa, X.~Zhou, and J.~X. Yu.
\newblock Duplicate-insensitive order statistics computation over data streams.
\newblock {\em IEEE Trans. Knowl. Data Eng.}, 22(4):493--507, 2010.

\end{thebibliography}

\end{document}